\documentclass[12pt]{extarticle}
\usepackage[utf8]{inputenc}
\usepackage[top = 2.5cm,left=3cm,right=2.5cm]{geometry}
\usepackage{graphicx}
\usepackage{amsmath}
\usepackage{amsfonts,mathrsfs}
\usepackage{amssymb}
\usepackage{soul}
\usepackage{color}
\usepackage{tikz}
\usepackage{stmaryrd}
\usepackage{ytableau}
\usepackage{float}
\usepackage[english]{babel}
\usepackage{ytableau}
\usepackage[symbol]{footmisc}
\usepackage{wasysym}
\usepackage{mathtools}
\usepackage{bbold}
\usepackage{caption}

\newtheorem{defn}{Definition}[section]
\newtheorem{theorem}{Theorem}[section]
\newtheorem{corollary}{Corollary}[section]
\newtheorem{lemma}{Lemma}[section]
\newtheorem{proposition}{Proposition}[section]
\newenvironment{proof}[1][Proof]{\noindent\textbf{#1.} }{\ \rule{0.5em}{0.5em}}

\newcounter{ourcount}
\newcounter{myenumi}
 \makeatletter
\newcommand*{\rom}[1]{\expandafter\@slowromancap\romannumeral #1@}
\makeatother

\newcommand{\leta}{\mathbb{0}}
\newcommand{\letb}{\mathbb{1}}
\newcommand{\letc}{\mathbb{2}}

\newcommand{\hA}{\mathcal{A}^*}

\newcommand{\cA}{\mathcal{A}}
\newcommand{\cB}{\mathcal{B}}
\newcommand{\cM}{\mathcal{M}}
\newcommand{\cF}{\mathcal{F}}
\newcommand{\com}{\Psi}
\newcommand{\GPZ}{\mathsf{PZ}}
\newcommand{\GZ}{\mathsf{Z}}
\newcommand{\ab}[1]{\rho_{#1}^{\mathrm{ab}}}

\newcommand{\cha}{\check{\alpha}}
\newcommand{\chb}{\check{\beta}}
\newcommand{\chg}{\check{\gamma}}

\renewcommand{\thefootnote}{\arabic{footnote}}

\newcounter{savefootnote}
\newcounter{symfootnote}
\newcommand{\symfootnote}[1]{%
   \setcounter{savefootnote}{\value{footnote}}%
   \setcounter{footnote}{\value{symfootnote}}%
   \ifnum\value{footnote}>8\setcounter{footnote}{0}\fi%
   \let\oldthefootnote=\thefootnote%
   \renewcommand{\thefootnote}{\fnsymbol{footnote}}%
   \footnote{#1}%
   \let\thefootnote=\oldthefootnote%
   \setcounter{symfootnote}{\value{footnote}}%
   \setcounter{footnote}{\value{savefootnote}}%
}

\begin{document}
\begin{center}
    {\Large \textbf{On balanced and abelian properties of circular words over a ternary alphabet}}

{\small
\vspace{.5cm} 
{\large \textbf{D.V. Bulgakova\symfootnote{dvbulgakova@gmail.com}, N. Buzhinsky$^{1}$, Y.O. Goncharov$^{2,3}$\symfootnote{yegor.goncharov@gmail.com} }}

\vskip .35cm 

\vskip .5cm $^{1}$ INSERM/Aix Marseille Universit\'{e}, \\ UMR1067, 13288, Marseille, France

\vskip .35cm $^{2}$ Service de Physique de l’Univers, Champs et Gravitation,
Universit\'e de Mons - UMONS,\\ 20 Place du Parc, B-7000 Mons, Belgique

\vskip .35cm $^{3}$ Institut Denis Poisson,
Universit\'e de Tours, Universit\'e d’Orl\'eans, CNRS,\\
Parc de Grandmont, 37200 Tours, France
}
    
\end{center}

\begin{abstract}
    We revisit the question of classification of balanced circular words and focus on the case of a ternary alphabet. We propose a $3$-dimensional generalisation of the discrete approximation representation of Christoffel words. By considering the minimal bound $3$ for abelian complexity of balanced circular words over a ternary alphabet, we provide a classification of all circular words over a ternary alphabet with abelian complexity subject to this bound. This result also allows us to construct an uncountable set of bi-infinite aperiodic words with abelian complexity equal to $3$.
\end{abstract}

\section{Introduction}
The notion of a balanced word originates from combinatorial studies of words \cite{L}, with its applications appearing in various areas where one needs to distribute objects as ``evenly'' as possible, for example, in algorithms dealing with synchronisation of processes and optimisation \cite{AGH,PV,S}, optimal scheduling \cite{G,MV}, construction of musical scales \cite{R}. 
Balanced words over a binary alphabet are well studied. The infinite aperiodic balanced words are called Sturmian words \cite{V,F} while the periodic ones are called Christoffel words \cite{C,BLRS}. For some applications it is more convenient to consider the latter as circular balanced words which can be imagined as finite words written along a circle and read periodically. 

While there is a complete classification of balanced circular words over a binary alphabet 
given in terms of Christoffel words, general classification of balanced circular words over arbitrary $N$-ary alphabets (which set we denote $[\cB_{N}]$) is lacking. The celebrated particular case is described by the so-called Fraenkel's conjecture \cite{Fraenkel,AGH} (proven to hold for $N\leqslant 7$) which states that there is unique, modulo isomorphism of alphabets, balanced circular word with pairwise different occurrences of letters, namely the so-called Fraenkel word. The following important partial result is also available: some of the circular balanced words over higher-$N$ alphabets can be constructed inductively starting from those over some lower-$N$ alphabet \cite{Gr,BCDJL}. Namely, consider a circular balanced word over $N$-ary alphabet with $n$ occurrences of some letter $a$ such that $n = km$ with some integers $m$ and $k>1$. Then one can denote $a_0 := a$ and introduce $k-1$ new letters $a_1,\dots,a_{k-1}$. By substituting $p$th occurrence of the letter $a$ in the initial circular word by $a_{p\;(\mathrm{mod}\;k)}$ one arrives at a new circular balanced word over $(N+k-1)$-ary alphabet. Applied at $N=3$, the words resulting from this procedure, supplemented with the Fraenkel word, give a complete classification of the set $[\cB_{3}]$. In order to compare the classification for $[\cB_{3}]$ with possible difficulties arising on a path towards describing $[\cB_{N}]$ for any $N$ we would like to draw the reader's attention to the work \cite{BCDJL} where particular classes of circular balanced words for $N=4,5,6$ were found numerically.

In the present work we propose a number of graphical constructions for the words $[\cB_{3}]$, except the Frankel word. First, there is an equivalent definition available for the circular words in question by generalising the well-known $2$-dimensional discrete approximation representation of Christoffel words \cite{L} to $3$ dimensions, with discrete walks demanded to approach particular planes from below without crossing. We also arrange the aforementioned set of words in a graph isomorphic to the Calkin-Wilf tree.

Along with the balanced property one can consider other characteristics of words based on the notion of abelian equivalence, see \cite{Pu}  and  references  therein. Two finite words are called abelian-equivalent if they can be obtained from each other by permutations of letters. The function which counts the number of classes of abelian-equivalent factors of a word is called the abelian complexity.  For a binary alphabet balanced circular words coincide with the set of all circular words with abelian complexity $\leqslant 2$. In this work we focus on words with abelian complexity $\leqslant 3$ over a ternary alphabet: first, we propose a classification of the circular words as above (which set we denote $[\cM_{3}]$) and show that they include balanced circular words as a proper subset. Meanwhile, for the ternary circular balanced words $[\cB_{3}]$ we distinguish those with abelian complexity exactly $3$. Second, we construct an uncountably infinite set of bi-infinite aperiodic words with abelian complexity $3$ which generalises the result of \cite{RSZ}. 

The paper is organised as follows. In Section \ref{sec:definitions} we recall the definitions concerning words, their properties and operations on them. In Section \ref{sec:classification} we recall the classifications for $[\cB_{2}]$ and $[\cB_{3}]$, give a classification for $[\cM_{3}]$  and construct (some of the) bi-infinite aperiodic words with abelian complexity $3$. In Section \ref{sec:geometry} we propose an equivalent classification of $[\cB_{3}]$ by generalising discrete approximations to a $3$-dimensional space. There we also propose a way of organising the set $[\cB_{3}]$ in a form of a binary tree. Technical proofs are placed in the appendix section \ref{sec:proofs}.

\section{Balanced and abelian properties of circular words}\label{sec:definitions}
\subsection{Alphabet and words}
Let $A_N = \{\leta,\letb,\letc,\dots\}$ be a $N$-ary alphabet\footnote{Despite decimal digits are available only for $N\leqslant 10$, the purpose of the current paper is covered by $N=2,3$.} supplemented by an order-preserving map $\iota:\{0,\dots, N-1\}\to A_N$. A {\it word} over an alphabet $A_N$ is an element of the free monoid $\hA_N$ generated by $A_N$, with the unit element $\varepsilon\in \hA_{N}$ referred to as the empty word. Each word $w$ is written as $w=a_1a_2\dots a_\ell$ for letters $a_1,a_2,\dots,a_{\ell}$. An integer $\ell\geqslant 0$ is called the length of $w$ and denoted by $|w|=\ell$. By definition, $\varepsilon$ is the word with zero length.

Let $\cA^{\ell}_{N}\subset \hA_{N}$ be a subset of words of length $\ell$. There is a decomposition of $\hA_{N}$ by the words' length:
\begin{equation*}
    \hA_{N} = \bigcup_{\ell = 0}^{\infty} \cA^{\ell}_{N}.
\end{equation*}

The monoid $\hA_N$ admits the following automorphisms. Let $\mathfrak{S}_N$ be the symmetric group whose elements are permutations of letters of $A_N$, i.e. for $\sigma\in\mathfrak{S}_N$, $\leta \to \sigma(\leta)$, $\letb\to\sigma(\letb)$, {\it etc.},  acting on $\hA_N$ as homomorphisms: a word $w=a_1\dots a_\ell\in\cA^{\ell}_N$ is mapped to $\sigma(w):=\sigma(a_1)\dots\sigma(a_{\ell})$ and $\sigma(\varepsilon) = \varepsilon$. Let $I$ denote invertion of a word, $I(a_1\dots a_{\ell}) = a_{\ell}\dots a_1$, and let $\GZ_{\ell}$ be an additive group acting  by cyclic permutations of letters in a word generated by $T(a_1a_2\dots a_{\ell}) = a_2\dots a_{\ell} a_1$. We denote by $\GZ$ the whole group of cyclic permutations of finite words of any length. The group $\GZ$ is a subgroup of a bigger group $\GPZ\supset \GZ$ generated by both $T$ and $I$. Actions of $\GPZ$ and $\mathfrak{S}_N$ mutually commute, and hence the whole group  of automorphisms under consideration is $G_N = \mathfrak{S}_N\times \GPZ$.

We define {\it circular words} as classes of words $[\hA_{N}]:=\hA_N\slash \GZ$ related by cyclic permutations of letters. For any representative $w\in\hA_N$ we denote the respective circular word as $[w]\in [\hA_{N}]$. We will say that $[w]\in[\hA_{N}]$ has length $\ell$ and write $[w]\in[\cA^{\ell}_{N}]$ if $|w| = \ell$. Two words $w$, $w^{\prime}$ belonging to the same class $[\hA_{N}]$, {\it i.e.} related by $\GZ$-action, are said to be conjugate.  Automorphisms of $[\hA_{N}]$ are given by the factor-group $[G_{N}]:=G_N\slash \GZ\cong \mathfrak{S}_N\times\mathbb{Z}_2$ (with the factor $\mathbb{Z}_2$ generated by the inversion $I$).
For a set of words $X\subset \hA_N$ define $[X]\subset[\hA_{N}]$ to be a set of all circular words containing representatives from $X$.

Let $H$ be a subgroup of $G_{N}$ (respectively, $[G_{N}]$). Then for any subset $X$ of $\hA_{N}$ (respectively, of $[\hA_{N}]$) notation $HX$ stands for the set of images of $X$ under the action of $H$. 

A word $u\in \hA_{N}$ is a {\it factor} of another word $w\in \hA_N$, denoted as $u\subset w$, if $w=vuv'$ for some words $v,\,v'\in \hA_N$. In particular, any word is a factor of itself. For two words $w,w^{\prime}\in \hA_N$ conjugate to each other there exist factors $u,v\in \hA_N$ such that $w=uv$ and $w'=vu$. A factor $u\in\hA_{N}$ of a circular word $[w]\in[\hA_{N}]$ is understood as $u\subset w^{\prime}\in [w]$ for some representative $w^{\prime}$ and denoted as $u\subset [w]$.

If for a word $w\in \hA_N$ there exists a factor $v\subset w$ such that $w = \underbrace{v\dots v}_{p>1}$, then we will say that $w$ is the $p$th power of $v$ and write $w = v^p$, while otherwise a word will be said to be {\it primitive}. 
Note that if a word is primitive, then so are its conjugates. Moreover, if a word $w$ is a $p$th power of some primitive factor $v$, then any word conjugate to $w$ is a $p$th power of a primitive factor conjugate to $v$. 

Due to that, the following definitions are correct in the sense of independence of choices of representatives: {\it i)} a circular word will be said to be primitive if it contains a primitive representative, {\it ii)} a circular word $[w]$ will be said to be a $p$th power of a primitive circular word $[v]$ (denoted by $[w] = [v]^p$) if for any representative $w^{\prime}\in [w]$ there is $v^{\prime}\in [v]$, such that $w^{\prime} = v^{\prime p}$. Note that $[v]^p = [v^p]$.


\subsection{Characteristics of distribution of letters}

Let $|w|_a$ denote the number of distinct occurrences of a letter $a$ in a word $w\in \hA_N$. To any word $w\in\hA_N$ we associate its {\it Parikh vector} which we will write in a form of a formal sum of letters from $A_N$ with integer coefficients:
\begin{equation*}
    \com(w) := |w|_{\leta}\,\leta + |w|_{\letb}\,\letb +\dots = \sum_{a\in A_N}|w|_{a}\,a\,.
\end{equation*}

The following characteristics commonly used in combinatorics on words describes distribution of letters within a word.
\begin{defn}\label{def:balanced}
    A non-empty word $w\in \hA_{N}$ is balanced if for each pair of factors $u,v\subset w $ such that $|u| = |v|$ we have
\begin{equation*}
    \com(u) - \com(v) = \delta_{\leta}\,\leta + \delta_{\letb}\,\letb + \dots\quad\text{with all}\quad |\delta_a|\leqslant 1.
\end{equation*}
A circular word $[w]\in[\hA_{N}]$ is balanced if every $w^{\prime} \in [w]$ is balanced.
\end{defn}
Note that because $\com(u)-\com(v)$ compares letter contents of two words of the same length, one has $\sum_{a\in A_N}\delta_a = 0$ in the above definition. 
\begin{defn}
    For a word $w\in \cA^{\ell}_N$ its $n$-spectrum (with $1\leqslant n \leqslant \ell$) is defined as 
    \begin{equation*}
        \mathsf{spec}_n w := \left\{\com(u)\;\middle|\;\text{for all}\; u\subseteq w\;\text{with}\; |u| = n\, \right\} \,.
    \end{equation*}
    For a circular word $[w]$ its spectrum $\mathsf{spec}_n [w]$ is a union of spectra of all representatives in the class:
    \begin{equation*}
        \mathsf{spec}_n [w] := \bigcup_{w^{\prime}\in [w]} \mathsf{spec}_n w^{\prime}.
    \end{equation*}
    Spectra of cardinality $1$ will be called trivial.
\end{defn}
With the definition of spectra at hand, uniformity of distribution of letters within a word can be described by the notion of abelian complexity \cite{Pu}. 

\begin{defn}\label{def:Myhill}
     The function $\ab{n}(w)=\# \mathsf{spec}_{n}w$ (respectively, $\ab{n}[w]=\# \mathsf{spec}_{n}[w]$) with $1\leqslant n\leqslant |w|$ is called the abelian complexity. If $p$ is a minimal integer such that for all $1\leqslant n\leqslant |w|$ we have $\ab{n}(w)\leqslant p$ (respectively, $\ab{n}[w]\leqslant p$) then we will say that the word $w$ (respectively, $[w]$) is abelian-$p$-bounded. 
\end{defn}

Note that $\rho^{\text{ab}}_{1}(w) \leqslant N$ is nothing else but the number of different letters from $A_N$ entering $w$. We will focus on those words that contain all of the $N$ letters by introducing a subset $\cA_N\subset \hA_N$ such that for any $w\in \cA_N$ we have $\rho^{\text{ab}}_{1}(w) = N$. 
We denote the set of respective circular words by $[\cA_{N}]$. We introduce the following notation $\cB_N\subset \cA_N$ for balanced words and $[\cB_{N}]\subset [\cA_{N}]$ for circular balanced words. The set $\cB_N$ is closed under the action of $[G_{N}]$, but not under the whole $G_N$. As for the set $[\cB_{N}]$, it is preserved by the action of $[G_{N}]$.
\vskip 0.2cm

\begin{lemma}\label{thm:BsubM}
    Let $K_{N} = \max_{k}\begin{pmatrix}N\\k \end{pmatrix}$. For a balanced word $w\in\cB_N$ (respectively, balanced circular word $[w]\in[\cB_{N}]$) of length $\ell$ we have $\ab{n}(w)\leqslant K_{N}$ (respectively, $\ab{n}[w]\leqslant K_{N}$) for all $1\leqslant n\leqslant \ell$.
\end{lemma}
\begin{proof}
    Note that for $w\in\cB_N$ (the case of a circular word $[w]\in[\cB_{N}]$ is treated along the same lines) for any two its factors $u,v\in w$ with $|u| = |v|$ we have
    \begin{equation}\label{eq:balanced_difference}
       \com(u) - \com(v) = a_1 + \dots + a_r - b_1 -\dots - b_r\,.
    \end{equation}
    such that all letters on the {\it rhs} are pairwise different.  Let us focus on the factors of a fixed length $n$ and choose some factor $u$ (with $|u| = n$). If for some factor $v$ one has $\com(u) - \com(v) = \varepsilon_a\,a + \dots$ with $\varepsilon = \pm 1$, then for any other factor $v^{\prime}$ one finds $\com(u) - \com(v^{\prime}) = \varepsilon_a k\,a + \dots$ with $k \in \{0,1\}$. Indeed, if $k = -1$ then $\com(v^{\prime}) - \com(v) = 2\varepsilon_a\,a + \dots$ in contradiction with the assumed balanced property. As a result, all letters in the alphabet are divided into tree classes: one of them is constituted by the letters that never appear on the {\it rhs} of \eqref{eq:balanced_difference}, while the rest is divided into two classes according to the values $\varepsilon_a$. Let the cardinalities of the latter be $N_{-}$, $N_{+}$ (without loss of generality we assume $N_{-} < N_{+}$). Because $\com(u)\in\mathsf{spec}_n w$ and any non-trivial {\it rhs} of \eqref{eq:balanced_difference} implies a contribution to the $n$-spectrum different from $\com(u)$, one has the following estimate:
    \begin{equation*}
        \#\mathsf{spec}_n w \leqslant \sum_{k=0}^{N_{+}}\begin{pmatrix} N_{+}\\ k  \end{pmatrix} \begin{pmatrix} N_{-}\\ k  \end{pmatrix} = \begin{pmatrix} N_{+} + N_{-}\\ N_{+}  \end{pmatrix}\,.
    \end{equation*}
    The above equality is due to the Vandermonde's identity. Recalling that $N_{+} + N_{-} \leqslant N$, one arrives at the highest estimate for the above bound to be $K_N$.
\end{proof}

\paragraph{Remark.} There is an interesting open question whether the bound in the above lemma can be improved or not. 
\vskip 0.2cm
 For a ternary alphabet considered in this work $K_{3} = 3$. We denote the set of abelian-$3$-bounded circular words by $[\cM_{3}]$.
\vskip 0.2cm

According to the following lemma, balanced circular words are fully classified by their primitive factors (the proof is straightforward).
\begin{lemma}\label{lem:building_blocks}
    For a primitive circular word $[w]\in [\cA_{N}]$, for any integer $p \geqslant 2$ the following assertions are equivalent:
    \begin{itemize}
        \item[1)] $[w]\in [\cA_{N}]$ is balanced,
        \item[2)] the $p$th power $[w^p]\in [\cA_{N}]$ is balanced.
    \end{itemize}
\end{lemma}
We denote the subsets of primitive words as $[\mathfrak{b}_{N}]\subset[\cB_{N}]$ and $[\mathfrak{m}_{N}]\subset[\cM_{N}]$.
\vskip 0.2cm
As an additional point to the above general part we bring two lemmas containing useful facts concerning abelian complexity.  Firstly, abelian complexity of different spectra of a circular word appear to be related.
\begin{lemma}\label{lem:spec_spec}
    For any $[w]\in [\cA^{\ell}_{N}]$ we have
    \begin{equation*}
        \ab{n}[w] = \ab{\ell-n}[w]\quad \text{for all}\quad 1\leqslant n< \ell\,.
    \end{equation*}
\end{lemma}

Another simplification about cardinalities of spectra of circular words comes in relation with their primitivity (the proof follows from \cite{CH}).
\begin{lemma}\label{lem:reducibility}
    For any $[w]\in [\cA^{\ell}_{N}]$ the two assertions are equivalent:
    \begin{itemize}
        \item[1)] there exist $1\leqslant n < \ell$ such that $\ab{n}[w]=1$, 
        \item[2)] $[w] = [u^p]$ for some $p>1$.
    \end{itemize}
\end{lemma}


\section{Classification of balanced\\ and abelian-$3$-bounded words over $A_{3}$}\label{sec:classification}
\subsection{Balanced circular words over $A_2$}\label{sec:Christoffel_def}

Before turning to the classification of the sets $[\cB_{3}]$ and $[\cM_{3}]$ we recall some known facts about balanced words and analogous classification for $A_2$.

Two-letter alphabet $A_2$ serves as a starting point where complete classification of balanced circular words is available and given by Christoffel words. Among a number of equivalent definitions of the latter we choose the following one giving a graphical representation as a discrete approximation of a line with a rational slope.
\begin{defn}\label{def:Christoffel}
    Let $q = (M-k)\slash k$ for some positive integers $k$ and $M>k$ such that $k$ and $M-k$ are coprime. Consider a line in $\mathbb{R}^2$ parametrised as $y=q\,x$ and a path which starts at the origin $(0,0)$ and is constructed by performing consecutive unit steps $\xi=(1,0)$ and $\eta = (0,1)$ such that: i) the step $\eta$ is performed always when it does not lead to going strictly above the line, ii) the whole path intersects the line twice. Then reading the steps consecutively as $\xi\to\leta$ and $\eta\to\letb$ leads to a word $C(k,M-k)$ over $A_2$ which is referred to as Christoffel word of a slope $q$.
\end{defn}

\noindent Note that, according to the above definition, $|C(k,M-k)| = M$, $|C(k,M-k)|_{\leta} = k$, $|C(k,M-k)|_{\letb} = M-k$. We denote the set of Christoffel words (respectively, circular Christoffel words) by $\mathfrak{c}$ (respectively, $[\mathfrak{c}]$) and the set of all their powers by $\mathcal{C}$ (respectively, $[\mathcal{C}]$).

\begin{figure}[H]
    \centering
    \includegraphics[width=0.24\textwidth]{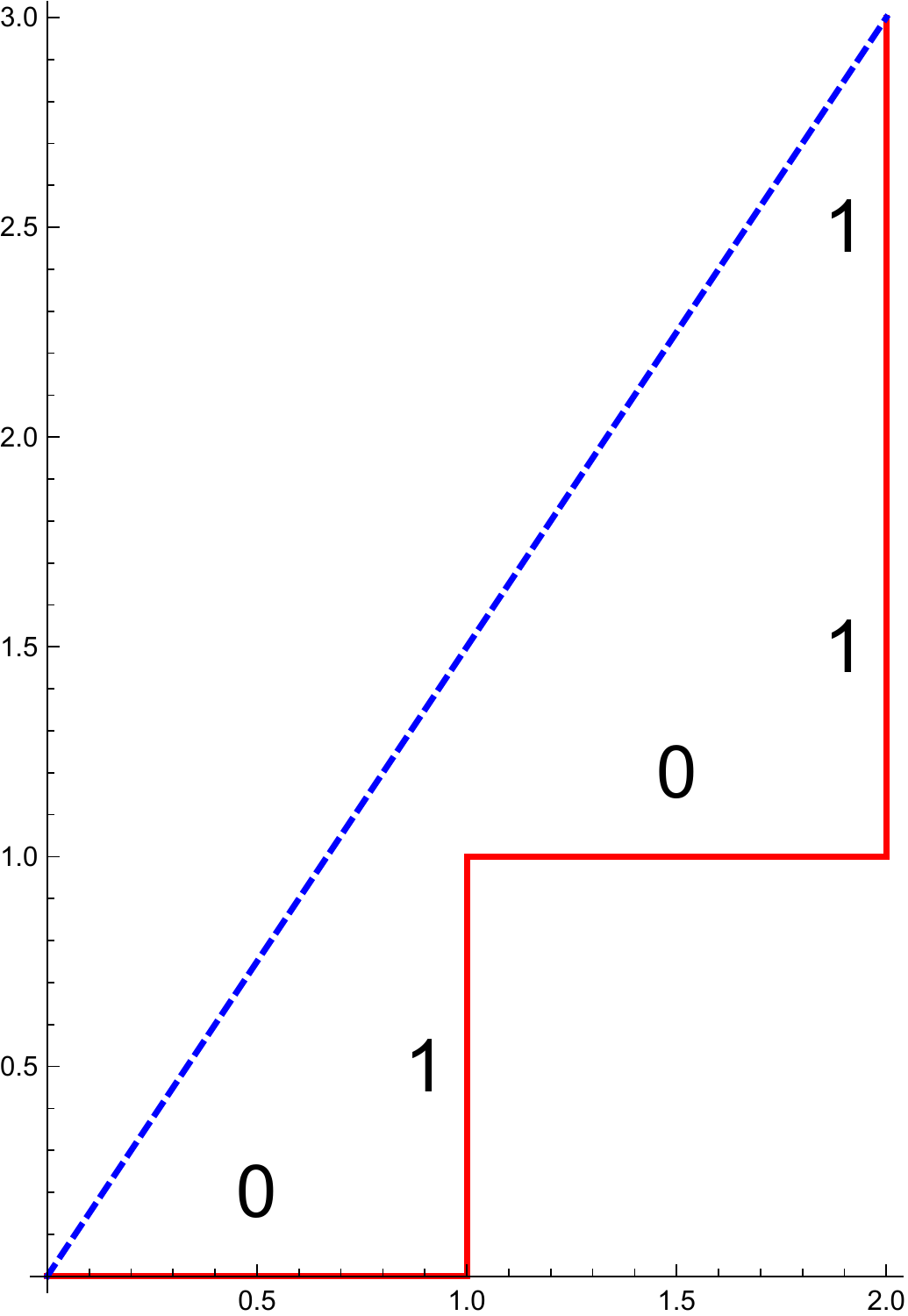}
    \hspace{0.1\textwidth}
    \includegraphics[width=0.24\textwidth]{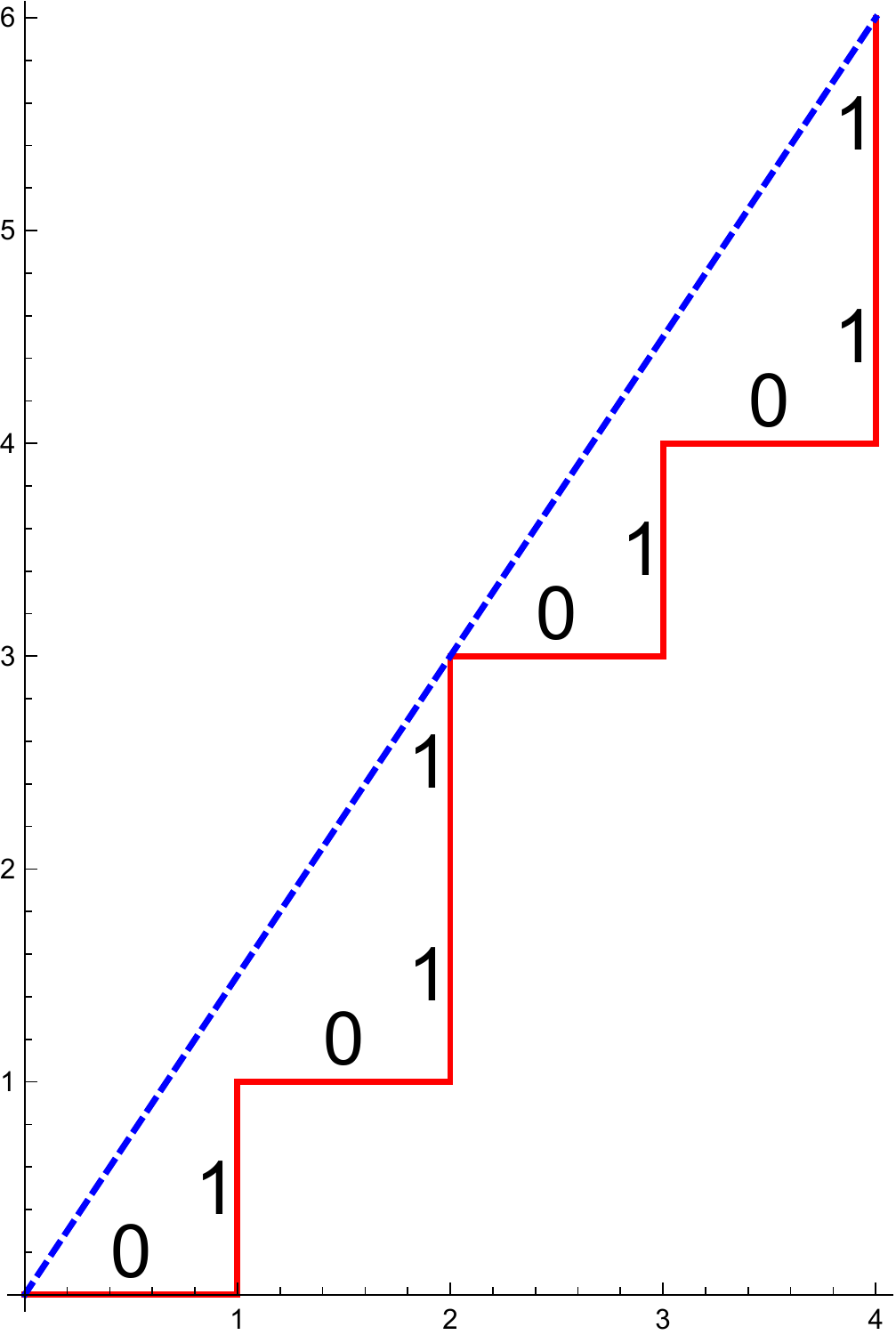}
    \caption{Graphical representation for Christoffel word $C(2,3) = \leta\letb\leta\letb\letb$ with the slope $q=3\slash 2$ and its square $C(2,3)^2 =  \leta\letb\leta\letb\letb\leta\letb\leta\letb\letb$.}
    \label{fig:my_label}
\end{figure}

The form of Christoffel words given by Definition \ref{def:Christoffel} is not preserved by $G_2=\mathfrak{S}_2\times \GPZ$.   Note that the so-defined Christoffel words are called the lower Christoffel words. Analogously, the upper Christoffel words parameterize the path that lies above the line segment and are given as inversion of lower Christoffel words. The set of lower Christoffel words is not invariant under the action of $\mathfrak{S}_2$. To show this, consider a permutation $\sigma_{(\leta\letb)}\in\mathfrak{S}_2$ such that $\leta\to\letb, \letb\to\leta$.  Then for a Christoffel word $C(2,3)= \leta\letb\leta\letb\letb$ we have $\sigma_{(\leta\letb)}\big( \leta\letb\leta\letb\letb\big)=\letb\leta\letb\leta\leta$, which does not meet the Definition \ref{def:Christoffel}. A remarkable feature of circular words $[\mathcal{C}]$ containing powers of lower Christoffel words as a representatives is that their set is closed under the action of $[G_{2}]=\mathfrak{S}_2\times\mathbb{Z}_2$. For the above example we get $[\sigma_{(\leta\letb)}\big( \leta\letb\leta\letb\letb\big)] = [\leta\leta\letb\leta\letb]$ with a Christoffel representative $C(3,2)$. In particular, the lower and upper Christoffel words are representatives of the same circular word. In the present paper we focus on circular words, and therefore working only with lower Christoffel words is sufficient for our purpose.

The following lemma shows that balanced circular words are exhausted by Christoffel representatives, see \cite{BLRS,CH}.
\begin{lemma}\label{lem:Christoffel_balanced_Myhill}
For $[w]\in [\cA_{2}]$ the following conditions are equivalent:
\begin{itemize} 
    \item[1)] there exists a Christoffel word $w^{\prime}$ such that $[w] = [w^{\prime p}]$ (for $p\geqslant 1$),
    \item[2)] $[w]$ is balanced,
    \item[3)] the abelian complexity $ \ab{n}[w]\leqslant 2$ for all $ 1\leqslant n\leqslant |w|$.
\end{itemize}
\end{lemma}

\vskip 0.2cm
We note the following well-known property of the Christoffel words which we will use hereafter, see \cite{BLRS}. Let $w\in\mathfrak{c}$ be a Christoffel word, then $w=\leta Q\letb$, where $Q$ is a {\it palindrome}\footnote{Palindromes are inversion-invariant words, {\it i.e.} those satisfying $I(Q) = Q$.}.
\begin{lemma}\label{lem:palindrome}
    A Christoffel word $w$ is the unique representative in the class $[w]$ which has the form $w=\leta Q\letb$ where $Q$ is a palindrome.
\end{lemma}

\subsection{Classification of balanced circular words over $A_3$}\label{sec:balanced_for_A3}

Let $\mathcal{C}^{\prime}\subset \mathcal{C}$ be a subset of powers of Christoffel words with even total number of $\leta$, and let a subset $\mathfrak{c}^{\prime}\subset\mathcal{C}^{\prime}$ be constituted by words of the form $C(k,M-k)^{p}$ with $p = 1$ for $k$ even and $p=2$ for $k$ odd. Along  the  lines  of \cite{Gr} consider a map $\phi:\mathcal{C}^{\prime}\to \mathcal{A}_{3}$ substituting each $\letb$ by $\letc$ and each {\it even} $\leta$ by $\letb$. For example, $\phi(\leta\leta\letb) = \leta\letb\letc$ and $\phi(\leta\letb\leta\letb) = \leta\letc\letb\letc$. Note that since any word in $\mathcal{C}^{\prime}$ necessarily contains at least one symbol $\letb$ and at least two symbols $\leta$, any $\phi$-image is indeed in $\cA_{3}$. The map $\phi$ respects powers of elements from $\mathcal{C}^{\prime}$: for any $c\in\mathcal{C}^{\prime}$ we have $\phi(c^p) = \phi(c)^p$.

\paragraph{Remark.} Alternatively, one could consider the set $\mathcal{C}^{\prime\prime}\subset \mathcal{C}$ of Christoffel words with even total amount of letters $\letb$ and consider a map $\phi^\prime:\mathcal{C}^{\prime\prime}\to \mathcal{A}_{3}$ substituting each even $\letb$ by $\letc$. Let us show that modulo $[G_{2}]$-action this choice is equivalent to considering the initially proposed set $\mathcal{C}^{\prime}$ and map $\phi$. As was noted in Section \ref{sec:Christoffel_def} the set of circular words $[\mathcal{C}]$ is closed under the action of $[G_{2}]$, in particular we have $\sigma_{(\leta\letb)}\big[\mathcal{C}^{\prime\prime}\big]=\big[\mathcal{C}^{\prime}\big]$, $\sigma_{(\leta\letb)}\in\mathfrak{S}_2$. Let $\sigma_{(\leta\letc\letb)}\in\mathfrak{S}_3$ be the permutation $\leta\to\letc,\,\letb\to\leta,\,\letc\to\letb$. It is straightforward to verify that $\sigma_{(\leta\letc\letb)}\big[\phi(\mathcal{C}^{\prime\prime})\big]=\big[\phi(\mathcal{C}^{\prime})\big]$, therefore $\mathfrak{S}_3\big[\phi(\mathcal{C}^{\prime\prime})\big]=\mathfrak{S}_3\big[\phi(\mathcal{C}^{\prime})\big]$.
\vskip 0.2cm

As a part of a classification of circular balanced words over $A_N$, the celebrated Fraenkel's conjecture says that for $N\geqslant 3$ there is a unique, up to $\mathfrak{S}_{N}$-action, primitive circular balanced word $\mathbf{F}_N$ (Fraenkel word) with pairwise distinct amounts of letters. It is constructed inductively as $\mathbf{F}_N = \mathbf{F}_{N-1}\iota(N)\mathbf{F}_{N-1}$ starting from $\mathbf{F}_1 = [\leta]$ and  has $\com(\mathbf{F}_N) = \sum_{j=0}^{N-1} 2^{j}\iota(j)$. For $N=3,\dots ,7$ the Fraenkel's conjecture is proven to hold, see \cite{AGH,Si,T}. Note that Fraenkel words are inversion-invariant.

Denote the set of $\mathfrak{S}_{3}$-images of the circular Fraenkel word over $A_3$ as $[\mathfrak{f}_{3}]$ and let $[\cF_{3}]$ stand for all their powers. The following theorem completely describes the set $[\cB_{3}]$ (for the proof of the first part of the assertion see \cite{Gr,BCDJL}).

\begin{theorem}\label{thm:main_1}
\begin{itemize}
    \item[1)] Balanced circular words over $A_3$ are given as
    \begin{equation*}
        [\cB_{3}] = \mathfrak{S}_3\big[\phi(\mathcal{C}^{\prime})\big]\sqcup [\mathcal{F}_{3}].
    \end{equation*}
    \item[2)] Primitive balanced circular words over $A_{3}$ are given as
    \begin{equation*}
        [\mathfrak{b}_{3}] =\mathfrak{S}_3[\phi(\mathfrak{c}^{\prime})] \sqcup[\mathfrak{f}_{3}] .
    \end{equation*}
\noindent
The upper bound $K_3$ is achived by all non-trivial spectra of $[\mathfrak{f}_{3}]$ and $\mathfrak{S}_3\big[\phi(C(k,M-k))\big]$ with $k$ even.
\end{itemize}

\end{theorem}

\subsection{Classification of circular abelian-$3$-bounded words over $A_3$}

In order to classify circular words with abelian complexity $\leqslant 3$ we define the {\it twisted words} constructed from powers of Cristoffel words as follows. Recall that any Christoffel word is of the form $\leta Q\letb$. For powers of $\leta Q \letb$ we define twisted words with interchanging $\letb\leta\to\leta\letb$ at (some of) the borders of primitive factors. For example from $C(2,1)^3 = \leta\leta\letb\,\leta\leta\letb\,\leta\leta\letb$ one can construct three twisted words: $\leta\leta\leta\,\letb\leta\letb\,\leta\leta\letb$, $\leta\leta\letb\,\leta\leta\leta\,\letb\leta\letb$ and $\leta\leta\leta\,\letb\leta\leta\,\letb\leta\letb$. We denote by $\mathcal{C}^{\text{tw}}$ the set of twisted words constructed from the set $\mathcal{C}^{\prime}$ of powers of Christoffel words with even number of zeros. 

For the following circular word $[\leta\letb\letc\letb\leta]$ denote its $\mathfrak{S}_3$-image by $[\mathfrak{d}_{3}]$ and all their powers by $[\mathcal{D}_{3}]$.

\begin{theorem}\label{thm:main_2}
For $A_{3}$, the set of circular abelian-$3$-bounded words is 
\begin{equation*}
    [\cM_{3}] = [\mathcal{B}_{3}]\sqcup \mathfrak{S}_3\big[\phi(\mathcal{C}^{\text{{\normalfont tw}}})\big]\sqcup[\mathcal{D}_{3}].
\end{equation*}
\end{theorem}

\subsection{Bi-infinite aperiodic abelian-$3$-bounded words over $A_3$}

Construction of the words $\phi(\mathcal{C}^{\text{{\normalfont tw}}})$ can be applied to obtain some infinite and bi-infinite aperiodic ternary words with abelian complexity $3$.  The following proposition generalises Theorem $4.3$ in \cite{RSZ}. 

Recall that a bi-infinite word $W = \dots a_{-1}a_0 a_{1}\dots$ is referred to as periodic if there is a positive integer $p$ such that $a_{i+p} = a_{i}$ for all $i\in \mathbb{Z}$. As a weaker property, $W$ is called ultimately periodic if there is $J\in\mathbb{Z}$ such that $a_{i+p} = a_{i}$ for all $i\geqslant J$. If $W$ is not ultimately periodic, it is called aperiodic.
\begin{proposition}\label{thm:infinite}
   Let $\omega\in \{\leta,\letb\}^{\mathbb{Z}}$ be an aperiodic bi-infinite word, and let $Q$ be a palindrome obtained as $\phi(C(m,n)) = \leta Q \letc$, $m$ even. 
   Then the image of $\omega$ under the morphism $\leta \to Q\letc\leta$, $\letb\to Q\leta\letc$ is abelian-$3$-bounded with abelian complexity $3$.
\end{proposition}
The proof of the above proposition follows from the Step $4$ in the proof of Theorem \ref{thm:main_2}. There is an immediate corollary for infinite words.
\begin{corollary}
    Let $W = \dots a_{-1} a_0 a_{+1}\dots $ be a bi-infinite ternary word obtained via the morphism described in Proposition \ref{thm:infinite}. Then, fixing any $i\in\mathbb{Z}$, the infinite word $W^{\prime} = a_{i}a_{i + 1}\dots$ is abelian-$3$-bounded with abelian complexity $3$.
\end{corollary}


\section{Geometrical constructions}\label{sec:geometry}

\subsection{Balanced words as $3$-dimensional\\ discrete approximations}

In this section we propose a geometrical way to construct all balanced words, except the Fraenkel word, by generalising discrete approximation representation for Christoffel words from Definition \ref{def:Christoffel} to $3$-dimensional space.

First, we define the notion of a $3$-dimensional discrete approximation for a pair of rational slopes $(q_1,q_2)$ as follows. Let points of a $3$-dimensional space $\mathbb{R}^{3}$ be parametrised as $(x,y,z)$. Starting from $(0,0,0)$ one constructs a path under a plane $z = q_1 x + q_2 y$ by performing unit steps $\xi = (1,0,0)$, $\eta = (0,1,0)$, $\zeta = (0,0,1)$. Step $\zeta$ is performed always if it does not lead to getting above the plane. Otherwise, one performs one of the steps $\xi,\eta$ such that they alternate along the path and step $\xi$ is made first. Procedure can be terminated at any point $(x^*,y^*,z^*)$ where $x^* = y^*$ and $z^* = q_1\,x^* + q_2\,y^*$. Such point always exists because numbers $q_1,q_2$ are rational.

Next, for any word $w\in\cA_{3}$ one can consider a line $\gamma(w)\subset\mathbb{R}^3$, which starts at the origin $(0,0,0)$ and proceeds by unit steps $\xi$, $\eta$, $\zeta$ for the letters $\leta$, $\letb$ and $\letc$ respectively. The line $\gamma(w)$ will be referred to as graphical representation for $w$.

\begin{figure}[H]
    \centering
    \includegraphics[width=0.4\textwidth]{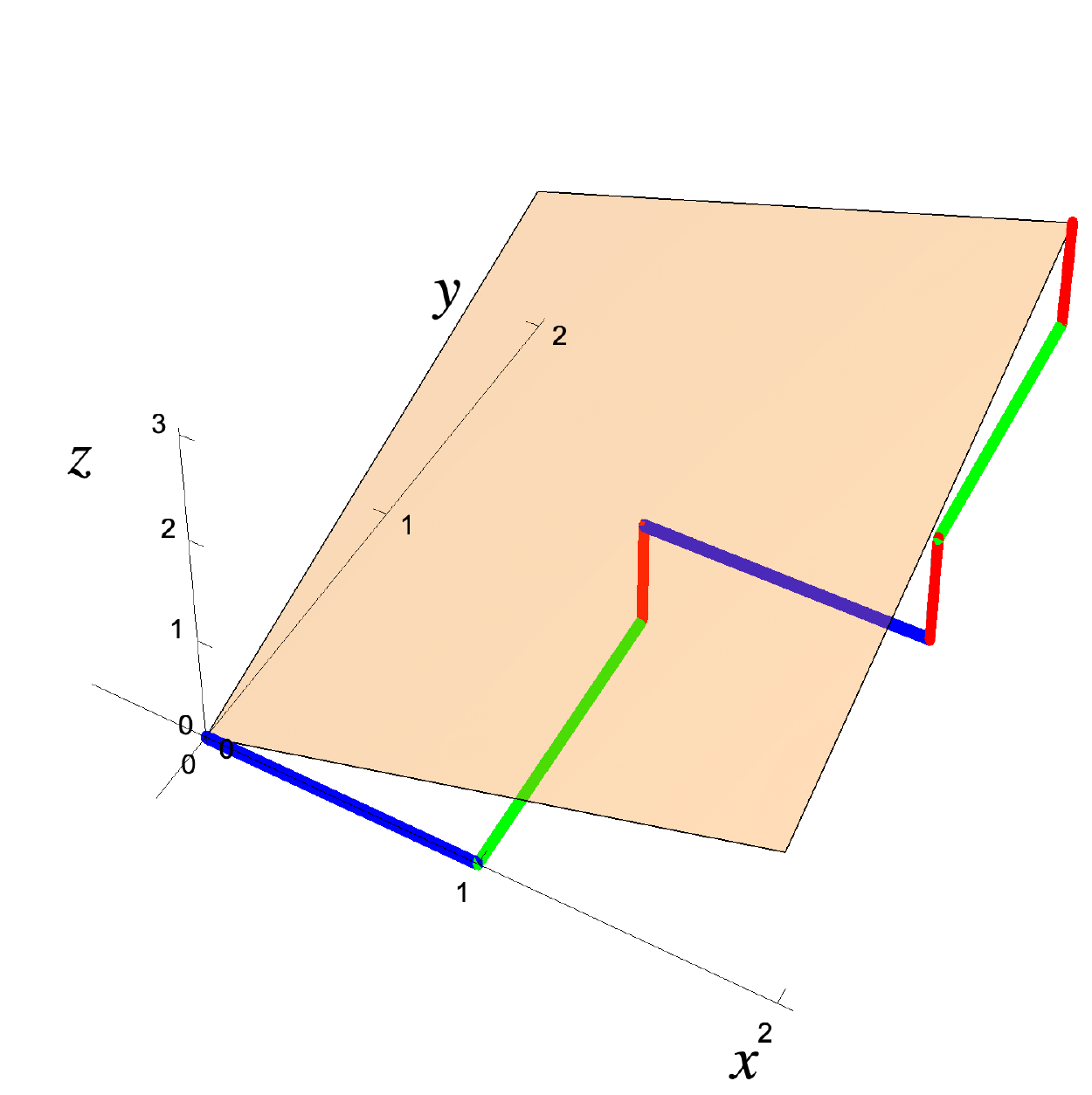}
    \caption{Graphical representation for $\phi(C(4,3)) = \leta\letb\letc\leta\letc\letb\letc$: blue (respectively, green and red) lines correspond to steps along $x$ (respectively, $y$ and $z$) read as $\leta$ (respectively, $\letb$ and $\letc$). The line $\gamma(\phi(C(4,3)))$ is a $3$-dimensional discrete approximation of a pair of slopes $(q,q)$ with $q=3\slash 4$.}
    \label{fig:representation_3d}
\end{figure}

The following theorem gives a graphical way of constructing the set $\phi(\mathcal{C}^{\prime})$ which sufficiently parametrises the set of balanced circular words $[\cB_{3}]\backslash[\mathcal{F}_{3}]$ by Theorem \ref{thm:main_1}.
\begin{theorem}
    For a ternary word $w\in\cA_{3}$ the following two assertions are equivalent:
    \begin{itemize}
    \item[1)] there is a Christoffel word such that $w = \phi(C(m,n)^{p})$, with $pm$ even,
    \item[2)] the graphical representation $\gamma(w)$ is a
    discrete approximation of pairs of rational slopes $(\tfrac{n}{m},\tfrac{n}{m})$ by reading step vectors as $\xi\to\leta$, $\eta\to\letb$, $\zeta\to\letc$ along the path.
    \end{itemize}
\end{theorem}

\begin{proof}
   First, we demonstrate that $2$-dimensional discrete approximation with the slope $q$ and $3$-dimensional discrete approximation for the pair $(q,q)$ can be constructed from one another. For a $3$-dimensional discrete approximation of a pair of rational slopes $(q,q)$ consider its orthogonal projection to the plane $x = y$. As soon as $z$-coordinate of any point is preserved under the projection, crossing the plane $z = q(x+y)$ by performing step $\zeta$ in $\mathbb{R}^3$ is equivalent to crossing the projection of the plane. More to that, each step $\xi$, $\eta$ increases the value of $z$ for the plane $z = q(x+y)$ by $q$, and the same holds for the projection. Finally, both $\xi$, $\eta$ are projected to the same horizontal step vector. If the length of the latter is scaled to be $1$ then $3$-dimensional discrete approximation of a pair $(q,q)$ is projected to a $2$-dimensional discrete approximation of a slope $q$ with even number of horizontal steps. Reverting all the steps, any $2$-dimensional discrete approximation of a rational slope $q$ with even number of horizontal steps can be turned to a $3$-dimensional discrete approximation of the pair $(q,q)$. 
    
    Consider $w = \phi(C(m,n)^p)$ with $p = 1$ (respectively, $2$) for $m$ even (respectively, odd). Let us verify that $q = n\slash m$ is the slope for a sought discrete approximation. Indeed, orthogonal projection of the graphical representation of $w$ to the plane $x = y$ leads to the graphical representation for $C(m,n)^r$ (because of the structure of the map $\phi$). Because the latter is a $2$-dimensional discrete approximation, we arrive at the conclusion that graphical representation of any $w\in\phi(\mathcal{C}^{\prime})$ is a $3$-dimensional discrete approximation.
    
    Other way around, any $3$-dimensional discrete approximation of a pair of rational slopes $(q,q)$ is projected to a $2$-dimensional discrete approximation of a slope $q$, which is equivalent to a word from $\mathcal{C}^{\prime}$.
\end{proof}

\subsection{Graph for balanced circular words}

Relation of balanced circular words over $A_3$ to Christoffel words over $A_2$, according to the Theorems \ref{thm:main_1} and \ref{thm:main_2}, allows us to arrange them in a binary tree as follows. Consider a graph of pairs of coprime numbers (Calkin–Wilf tree), then substitute each pair $(m,n)$ at each vertex by a triple $(m,m,n)$.
\begin{figure}[H]
    \centering
    \includegraphics[width=0.9\textwidth]{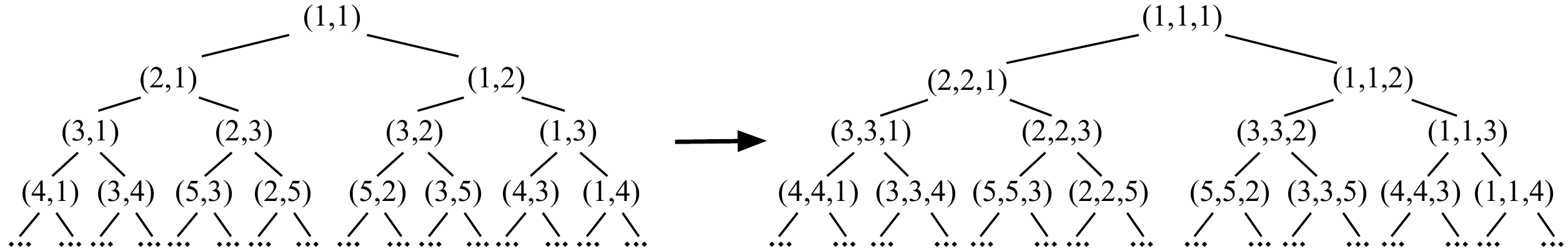}
    \caption{A graph of pairs of coprime numbers.}
    \label{fig:graph_primes}
\end{figure}
Each coprime triple $(m,m,n)$ implies a word $[w]\in[\cB_{3}]$ such that $|w|_{\leta} = |w|_{\letb} = m$ and $|w|_{\letc} = n$. Recall that any word $[w]\in[\cB_{3}]\backslash [\cF_{3}]$ is constructed as a $\phi$-image either of $C(2m,n)$ (for odd $n$) or $C(m,n^{\prime})C(m,n^{\prime})$ (in this case assign $n = 2n^{\prime}$) which allows us to arrange the set $[\cB_{3}]\backslash [\cF_{3}]$ in a graph presented on fig.~\ref{fig:graph_numb} (with representatives fixed up to cyclic permutations and $\mathfrak{S}_3$-action). To our knowledge, pairs of words joined by edges are not related by a morphism.
\begin{figure}[H]
    \centering
    \includegraphics[width=0.9\textwidth]{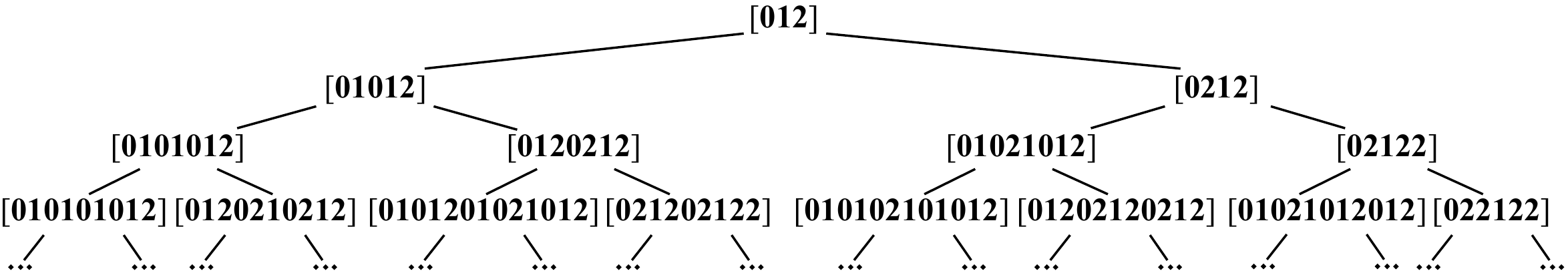}
    \caption{}
    \label{fig:graph_numb}
\end{figure}
Note that $(2m,n)$ (for $n$ odd) and $(m,n^{\prime})$ (for $n=2n^{\prime}$) are indeed pairs of coprime numbers provided that $(m,n)$ is a coprime pair. All words from $[\cB_{3}]\backslash [\cF_{3}]$ (modulo interchanging letters by $\mathfrak{S}_3$) indeed enter the graph because to any Christoffel word $C(m,n)$ there corresponds a pair of coprime numbers $(\frac{m}{2},n)$ (for $m$ even) or $(m,2n)$ (for $m$ odd) belonging to the left graph on fig \ref{fig:graph_primes}. 

In order to make inversion symmetry of the words from $[\cB_{3}]$ manifest, as well as ``factor out'' the action of $\mathfrak{S}_3$, we consider the following illustration. A circular word of length $\ell$ is represented by a graph with $\ell$ vertices placed on an oriented circle. Letters are mapped to vertices one-by-one such that each next letter is mapped to the next vertex. Edges join vertices in a way that any maximal subset of vertices corresponding to the same letter become vertices of a polygon. 
\begin{figure}[H]
    \centering
    \includegraphics[width=0.8\textwidth]{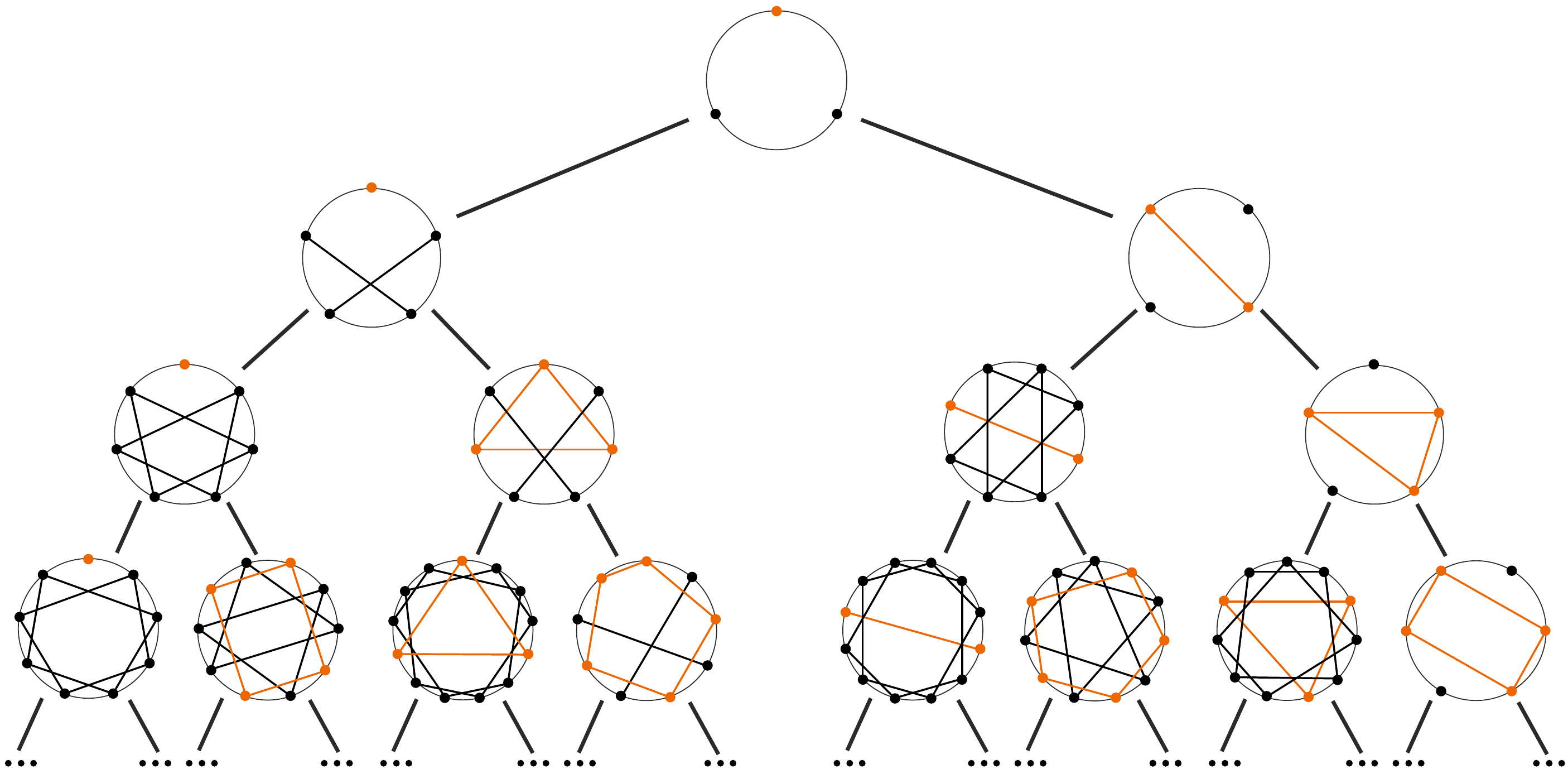}
    \caption{Illustration for the words from $[\cB_{3}]\backslash[\cF_{3}]$ trivialising $\mathfrak{S}_{3}$-action and making inversion symmetry manifest. Colouring of vertices and edges is made to visually separate alternating letters $\leta$, $\letb$ from $\letc$ in $\phi(\mathcal{C}^{\prime})$.}
    \label{fig:graph}
\end{figure}

\section*{Acknowledgements}
We are grateful to Anna Frid for enlightening discussions and instructive comments during preparation of the manuscript. The work of Y.G. is supported by a joint grant ``50/50'' UMONS -- Universit\'e Fran\c{c}ois Rabelais de Tours.

\appendix
\section{Proofs}\label{sec:proofs}

\begin{proof}[Proof of Theorem \ref{thm:main_1}]
Let  $C(k,M-k)\in\mathcal{C}^{\prime}$ be primitive (a Christoffel word with $k$ even). The respective circular word $[C(k,M-k)]$ is balanced with abelian complexity $\leqslant 3$ due to the Lemma \ref{lem:Christoffel_balanced_Myhill}, and hence its spectra are of the form $\mathsf{spec}_{n}[C(k,M-k)]=\{p\,\leta+(n-p)\letb,\,(p+1)\leta+(n-p-1)\letb\}$ with $p\leqslant k$. Suppose $p$ is even (similar arguments hold for $p$ odd). Then for the $\phi$-image $\mathsf{spec}_{n}[\phi(C(k,M-k))]$ is constituted by $\frac{p}{2}\leta+\frac{p}{2}\letb+(n-p)\letc$ and at least by one of the elements $(\frac{p}{2}+1)\leta+\frac{p}{2}\letb+(n-p-1)\letc$ or $\frac{p}{2}\leta+(\frac{p}{2}+1)\letb+(n-p-1)\letc$. Lemma \ref{lem:reducibility} implies that $[\phi(C(k,M-k))]$ is primitive.

Let $C(k,M-k)$ be a Christoffel word with $k$ odd. The $n$-spectra of $[\phi\big(C(k,M-k)C(k,M-k)\big)\big]$ for $n\neq M$ have similar form to the words with $k$ even described above. According to Lemma \ref{lem:reducibility} we have $\mathsf{spec}_{M}[C(k,M-k)C(k,M-k)]=\{k\leta+(M-k)\letb\}$. After applying the map $\phi$ we get $\mathsf{spec}_{M}[\phi(C(k,M-k)C(k,M-k))]=\{\frac{k-1}{2}\leta+\frac{k+1}{2}\letb+(M-k)\letc,\,\frac{k+1}{2}\leta+\frac{k-1}{2}\letb+(M-k)\letc\}$. Therefore, the words $[\phi\big(C(k,M-k)C(k,M-k)\big)]$ with $k$ odd are primitive by the Lemma \ref{lem:reducibility}. In addition the $M$-spectrum does not achieved the upper bound $K_3=3$. 

Finally, we will prove that non-trivial spectra of the words $[\phi\big(C(k,M-k)\big)]$ with $k$ even saturate the upper bound $K_3$, {\it i.e.}  $\#\mathsf{spec}_n\big[\phi(C(k,M-k))\big] = 3$ for all $1\leqslant n \leqslant M-1$. Recall that $\#\mathsf{spec}_1\big[\phi(C(k,M-k))\big] = 3$ already because $\phi$ is a map to $\cA_{3}$. If we suppose that for some $n > 1$ it happens that $\#\mathsf{spec}_{n}\big[\phi(C(k,M-k))\big] < 3$, then, without loss of generality and according to the aforesaid, $\mathsf{spec}_{n}\big[\phi(C(k,M-k))\big] = \left\{\frac{p}{2}\leta+\frac{p}{2}\letb+(n -p)\letc,\frac{p}{2}\leta+(\frac{p}{2}+1)\letb+(n -p-1)\letc\right\}$ for some even $p<n$. Such spectrum tells that any factor of length $\ell$ contains always the same amount $\tfrac{p}{2}$ of symbols $\leta$. This implies for the preimage $C(k,M-k)$ that for each $\leta$ situated at an odd position (less than $M$) in $C(k,M-k)^2$ there is another $\leta$ situated $n - 1$ steps after. Therefore for factors of the form $\leta P \letb \subset C(k,M-k)^2$ with $|P| = n - 1$ the position of $\leta$ is necessarily even. Due to Lemma \ref{lem:palindrome} $C(k,M-k) = \leta Q \letb$, where $Q$ is a palindrome with $k-1$ (odd) amount of $\leta$. Applying the above arguments to the first occurrence of $\leta$ we can write $C(k,M-k) = \leta P\leta Q^{\prime} \leta I(P)\letb$, where all explicitly written $\leta$ occur at odd positions. The obtained situation is in contradiction to the above arguments. 
\end{proof}
\vskip 0.7cm

\begin{proof}[Proof of Theorem \ref{thm:main_2}]
    Because the set $[\cB_{3}]$ is completely described by Theorem \ref{thm:main_1} we focus on words with abelian complexity $\leqslant 3$ that are not balanced. As a matter of convenience, the whole proof will be divided in a number of steps. 
    \paragraph{Step 1: elementary factors.} For any word $[w]\in[\cM_{3}]\backslash[\cB_{3}]$ there is a scale $\ell\geqslant 1$ such that $\mathsf{spec}_{\ell}[w] = \{\alpha,\beta,\gamma\}$ with $\beta - \alpha = a - c$ and $\gamma - \beta = b - c$ ($a,b\neq c$). This implies $\gamma - \alpha = a + b - 2c$ which causes $[w]\notin[\cB_{3}]$, and therefore we will say that $\ell$ is a non-balanced scale. Further, let $\ell$ represent the minimal non-balanced scale. It is convenient to map the letters of $[w]$ to elements of the $\ell$-spectrum such that the latter are obtained from factors starting from the corresponding letters. This gives rise to a circular word $[W_{\ell}]$ over the alphabet $\{\check{\alpha},\check{\beta},\check{\gamma}\}$, with the letters corresponding to the elements of $\mathsf{spec}_{\ell}[w]$. In this respect, there exists a factor $\cha\chb^k\chg\subset [W_{\ell}]$. This sequence is realised by a factor of the form $\underbracket[0.5pt]{cQcU }_{\ell}\underbrace{aQcU\dots aQcU}_{m\geqslant 0\text{ factors }aQcU} aQb\subset [w]$ with some $Q,U\in\hA_{N}$. This determines $\mathsf{spec}_{\ell}[w] = \{\alpha=2c+\com(Q)+\com(U),\beta=a+c+\com(Q)+\com(U),\gamma=a+b+\com(Q)+\com(U)\}$. Minimality of $\ell$ also fixes $U = \varepsilon$ because $\com(aQb) - \com(cQc) = a + b - 2c$, leading to the factor \begin{equation}\label{eq:factor}
        f_m^{(\ell)}(Q) = \underbracket[0.5pt]{cQc}_{\ell}\underbrace{aQc\dots aQc}_{m\geqslant 0\text{ factors }aQc}aQb.
    \end{equation}
Particular form of $Q$ is not specified at this stage, however its length $|Q| = \ell - 2$ is unambiguously related to the minimal non-balanced scale $\ell$. The two cases whether $Q$ is empty or not lead to different situations.

Let $Q = \varepsilon$ ({\it i.e.} $\ell = 2$). Let us show that one necessarily has $a\neq b$. Indeed, assuming the opposite and using that $[w]$ is circular and contains each letter at least once, one arrives at $aa\dots h_1 b^{n} h_2\dots cc\subset [w]$ with $h_1,h_2\in\{a,c\}$ and $n\geqslant 1$. If $h_1\neq h_2$ then $\mathsf{spec}_2[w]\ni 2a, 2c, b+h_1, b+h_2$, which contradicts $[w]\in[\cM_{3}]$. Let (without loss of generality) $h_1 = h_2 = a$. Then $\mathsf{spec}_2[w]$ contains $2a,2c,a+b$ and also at least one of the $a+c$ or $b+c$, which again contradicts $[w]\in[\cM_{3}]$. 

Thus, we have $f^{(2)}_m(Q) = cc(ac)^{m}ab$ ({\it i.e.} $Q = (ac)^m$). If $m\neq 0$ then necessarily $\mathsf{spec}_2[w] = \{2c, a+b, a+c\}$, which requires that $b$ enters $[w]$ only with a factor $aba$. But in this case $\mathsf{spec}_3[w] \ni a+2c,2a+c,a+b+c,2a+b$, which contradicts $[w]\in[\cM_{3}]$ and we are left with $f^{(2)}(Q)_0 = ccab$. As above, $\mathsf{spec}_2[w]$ is fixed and requires that $f_0^{(2)}(Q)a =  ccaba\subset [w]$. This in turn fixes $\mathsf{spec}_3[w] = \{a+2c,2a+b,a+b+c\}$. Together $\mathsf{spec}_2[w]$ and $\mathsf{spec}_3[w]$ require that $[w] = [\underbrace{w_0\dots w_0}_p]$ with the primitive factor $[w_0] = [ccaba]$. For any pairwise-different $a,b,c$ the so obtained primitive words are in $[\mathfrak{d}_{3}]$ (and all their powers constitute $[\mathcal{D}_{3}]$).

Let $Q\neq\varepsilon$. We demonstrate that condition $a\neq b$ leads to an inconsistency. Namely, factor \eqref{eq:factor} fixes $\mathsf{spec}_{\ell}[w]$ as $\alpha = 2c + \com(Q)$, $\beta = a + c + \com(Q)$ and $\gamma = a + b + \com(Q)$. If $m\neq 0$, then $\mathsf{spec}_{\ell+1}[w] = \{2a+c+\com(Q),a+2c+\com(Q),a+b+c+\com(Q)\}$, which requires $f^{(\ell)}_m(Q)\,c\subset [w]$. This implies in turn that a new element $b+c+\com(Q)\in\mathsf{spec}_{\ell}[w]$, leading to a contradiction with $[w]\in[\cM_{3}]$. We are left with $f^{(\ell)}_0(Q) = cQcaQb$. The spectrum $\mathsf{spec}_{\ell}[w]$ fixes the extension $f^{(\ell)}_0(Q) a\subset [w]$. Extended factor allows us to reconstruct $\mathsf{spec}_{\ell+1}[w] = \{a+2c+\com(Q), 2a+b+\com(Q), a+b+c+\com(Q)\}$. This restricts extension to the left to $af^{(\ell)}_0(Q)a\subset [w]$.Let $Q = \bar{Q}y$. The cases $y\in\{b,c\}$ lead to a non-balanced scale $\ell^{\prime} = 2 < \ell$ within $af_0^{(\ell)}(Q)a$, hence $y=a$. But then $\mathsf{spec}_{\ell+1}[w]\ni 3a+b+\com(\bar{Q}),3a+c+\com(\bar{Q}),2a+b+c+\com(\bar{Q}), 2a+2c+\com(\bar{Q})$ is in contradiction to $[w]\in[\cM_{3}]$.  Therefore necessarily $a = b$.

We arrive at the factor
\begin{equation}\label{eq:factor_Q}
    f^{(\ell)}_m(Q) = cQc\underbrace{aQc\dots aQc}_{m\geqslant 0\text{ factors }aQc}aQa \subset [w]
\end{equation}
which realises the sequence of the elements of $\ell$-spectrum $\cha\chb^{(m+1)\ell - 1}\chg\subset [W_{\ell}]$. As for a sequence $\chg\chb^{(m^{\prime}+1)\ell - 1}\cha$, it is realised by a factor of the form
\begin{equation}\label{eq:bfactor_Q}
    \bar{f}^{(\ell)}_{m^{\prime}}(Q^{\prime}) = aQ^{\prime}a\underbrace{cQ^{\prime}a\dots cQ^{\prime}a}_{m^{\prime}\geqslant 0\text{ factors }cQ^{\prime}a}cQ^{\prime}c \subset [w],\quad |Q^{\prime}| = \ell-2,
\end{equation}
which follows by repeating all steps from the very beginning. As a result, any sequence $\cha\chb^{(m+1)\ell - 1}\chg$ (respectively, $\chg\chb^{(m+1)\ell - 1}\cha$) in $[W_{\ell}]$ is realised by an elementary factor $f_m^{(\ell)}(Q)$ (respectively, $\bar{f}_m^{(\ell)}(Q)$) with some particular $m\geqslant 0$ and $|Q|=\ell-2$.

Finally, note that $Q$ in \eqref{eq:factor_Q} (similarly for $Q^{\prime}$ in \eqref{eq:bfactor_Q}) necessarily contains $b\neq a,c$: this is because $\beta - \alpha = \gamma - \beta = a - c$ and hence $[w] = [bP_1bP_2\dots bP_s]$ with $|P_1|=\dots=|P_s| = \ell - 1$, while $|f^{(\ell)}_m(Q)| > \ell - 1$.

\paragraph{Step 2: universality of elementary factors.} The structure of elementary factors \eqref{eq:factor_Q} and \eqref{eq:bfactor_Q} of $[w]\in[\cM_{3}]\backslash[\cB_{3}]$ restricts $Q$ to be universal, {\it i.e.} the same for all elementary factors of the form \eqref{eq:factor_Q} and $Q^{\prime} = Q$ for all elementary factors of the form \eqref{eq:bfactor_Q}. First, note that $\com(Q)$ is universal: this is verified for any two elementary factors by expressing $\alpha = 2a + \com(Q) = 2a + \com(Q^{\prime})$. Next, $f_m^{(\ell)}(Q)\subset [w]$ implies that $[cQa]\in[\cB_{3}]$ (factors \eqref{eq:bfactor_Q} are considered along the same lines). This follows immediately by checking that $\mathsf{spec}_{t}[cQa]\subset\mathsf{spec}_tf_m^{(\ell)}(Q)$ for all $t<\ell$ and taking into account that $\ell$ is the minimal non-balanced scale of $[w]$. From the Theorem \ref{thm:main_1} it follows that $[cQa] \in\mathfrak{S}_3\big[\phi(\mathcal{C}^{\prime})\big]$.

Moreover, it appears that $cQa \in \mathfrak{S}_3\left(\phi(C(k,M-k))\right)$ with $k$ even. To prove this we first show that $I(Q) = Q$ by induction. For the base suppose $Q = x Q_1 x^{\prime}$ such that $\mathsf{spec}_{2}[w]\ni a+c$,  $a+x$, $c+x$, $a+x^{\prime}$, $c+x^{\prime}$. Assuming $x\neq x^{\prime}$ contradicts either $[w]\in[\cM_{3}]$ or minimality of the non-balanced scale $\ell$. For the step of induction suppose $Q = V Q_r I(V)$ with $|V| = r$ for some $r\geqslant 1$. Then assuming $Q_r = x Q_{r+1} x^{\prime}$ implies $\mathsf{spec}_{r + 2}[w] \ni a+c+\com(V)$, $a+x+\com(V)$, $c+x+\com(V)$, $a+x^{\prime}+\com(V)$, $c+x^{\prime}+\com(V)$, which is inconsistent either with $[w]\in[\cM_{3}]$ or with minimality of the non-balanced scale $\ell$ unless $x = x^{\prime}$. As a result, $Q = I(Q)$ for any elementary factor \eqref{eq:factor_Q} of \eqref{eq:bfactor_Q}. This implies in turn that $cQa = \sigma(\phi(C(k,\ell-k)))$ with $k$ even and $\sigma\in\mathfrak{S}_3$ such that  $\sigma(\leta) = c$, $\sigma(\letb) = b$ and $\sigma(\letc) = a$. Indeed, among the representatives of $[C(k,\ell-k)^p]$ only at $p = 1$ only $C(k,\ell-k)$ itself is of the form $\leta S \letb$ with $I(S) = S$. Finally, $I(Q) = Q$ requires $k$ to be even, because only in this case $\phi(\leta S\letb) = \leta\, \sigma^{-1}(Q)\,\letc$ contains a palindrome $\sigma^{-1}(Q)$. 

Recall that $\com(Q)$ is universal in the above sense, and therefore for any $Q$ in \eqref{eq:factor_Q} or \eqref{eq:bfactor_Q} the word $cQa$ is a $\sigma\circ\phi$-image of the same Christoffel word $C(k,\ell-k)$ leading to the universality of $Q$. We arrive at the following important conclusion: factors $f^{(\ell)}_m(Q)$ are $\sigma\circ\phi$-images of twisted words from $\mathcal{C}^{\text{tw}}$ with $\letb\leta\to\leta\letb$ at all borders between primitive factors, while $\bar{f}^{(\ell)}_{m}(Q) = I\left(f^{(\ell)}_m(Q)\right)$. 

Note also the restriction following from the structure of a $\phi$-image: for $b\neq a,c$ any factor of the form $cPc$ in \eqref{eq:factor_Q} or \eqref{eq:bfactor_Q} is such that $|P|_b\neq 0$, in particular factor $cc$ is forbidden.

\paragraph{Step 3: merging of elementary factors.} At this step we demonstrate that sequences $\cha\chb^{k_1}\chg\chb^{k_2}\cha\chb^{k_3}\chg\dots \chb^{k_{2t}}\subset [W_{\ell}]$ with $t\geqslant 1$ and $k_1,\dots,k_{2t} \geqslant 1$ and are realised by alternating factors \eqref{eq:factor_Q} and \eqref{eq:bfactor_Q} properly ``merged'' to one another.

First, let us demonstrate that \eqref{eq:factor_Q} and \eqref{eq:bfactor_Q} are uniquely extended to
\begin{equation}\label{eq:f_extensions}
    f_{m}^{(\ell)}(Q)\subset a  f_{m}^{(\ell)}(Q) c\subset [w]\quad \text{and}\quad \bar{f}_{m^{\prime}}^{(\ell)}(Q)\subset c  \bar{f}_{m^{\prime}}^{(\ell)}(Q) a\subset [w],
\end{equation}
{\it i.e.} any sequence $\cha\chb^{m}\chg$ for \eqref{eq:factor_Q} (respectively, $\chg\chb^{m^\prime}\cha$ for \eqref{eq:bfactor_Q}) is extended to $\chb\cha\chb^{m}\chg\chb$ (respectively, $\chb\chg\chb^{m^\prime}\cha\chb$). We present how extension $f_{m}^{(\ell)}(Q)\subset  f_{m}^{(\ell)}(Q) c\subset [w]$ is fixed, while all the rest is considered along the same lines.

Recall that factor \eqref{eq:factor_Q} realises the sequence $\cha\chb^{(m+1)\ell - 1}\chg$, which can be extended to the right either by $\chb$ or $\chg$. Supposing the latter leads to extension $f^{(\ell)}_m(Q)a\subset [w]$. This means $\mathsf{spec}_{2}[w]\ni 2a, a+c$. As a result of the Step 1, for $b\neq a,c$ there is $|w|_b\neq 0$ what fixes $\mathsf{spec}_{2}[w] = \{2a,a+c,a+b\}$ because other possibilities would lead to a contradiction with minimality of the non-balanced scale $\ell > |Q|$. Therefore, we can fix $Q=aQ_1a$. We proceed inductively by supposing that $\mathsf{spec}_{r+1}[w] = \{(r+1)a, c + r\,a, b + r\,a\}$ and $Q = a^{r} Q_r a^{r}$ for $r\geqslant 1$ (where $Q_0 = Q$). Indeed, the form $Q = a^r Q_r a^r$ demands $\mathsf{spec}_{r+2}[w]\ni (r+2)a$, $c+(r+1)a$. There is necessarily another element $\mathsf{spec}_{r+2}[w]\ni b+(r+1)a$ fixed by minimality of the non-balanced scale $\ell >|Q|$. And finally $\mathsf{spec}_{r+2}[w]$ fixes $Q_r = aQ_{r+1}a$. Induction leads to $Q = a^{\ell - 2}$ which contradicts $|Q|_b \neq 0$. All in all we arrive at the uniqueness of the extension $f_{m}^{(\ell)}(Q)\subset  f_{m}^{(\ell)}(Q) c\subset [w]$, and by similar treatment to \eqref{eq:f_extensions}.

Next, let us show that extension of the factor $a f_m^{(\ell)}(Q) c\subset [w]$ to the right preserves $Q$ in a sense that the following factor is impossible:
\begin{equation}\label{eq:f_extension_Q}
    a f_m^{(\ell)}(Q) \underbrace{cQa\dots cQa}_{n\geqslant 0 \text{ factors } cQa} c \widetilde{Q} \subset [w]
\end{equation}
with $|\widetilde{Q}|=|Q|$ but $\widetilde{Q}\neq Q$. Recalling that $a+c+\com(Q) = \beta$, if $Q = P x P^{\prime}$ and $\widetilde{Q} = P \widetilde{x} \widetilde{P}^{\prime}$ (with a letter $\widetilde{x}\neq x$) then either $x = a$, $\widetilde{x} = c$ realising $\chb\cha \subset[W_{\ell}]$ or $x = c$, $\widetilde{x} = a$ realising $\chb\chg \subset[W_{\ell}]$. The former case leads to factors $aPa,c Pc\subset [w]$, which contradicts minimality of the non-balanced scale $\ell$, so $x = c$ and $\widetilde{x} = a$. Proceeding one step to the right, one fixes $Q=PcaP^{\prime\prime}$ and $\widetilde{Q} = P ac \widetilde{P}^{\prime\prime}$. Indeed, to avoid contradiction with balanced property we necessarily have either $Q=PcbP^{\prime\prime}$ or $Q=PcaP^{\prime\prime}$. The former case leads to a contradiction with Myhill property because $\mathsf{spec}_{|P|+3}[w]\ni 2c+b+\com (P), 2c+a+\com (P), 2a+c+\com (P), a+b+c+\com (P)$. By similar treatment we get $\widetilde{Q} = P ac \widetilde{P}^{\prime\prime}$. By continuing the above arguments we assume the following form of factors $Q=P_0caP_1ca\dots caP_k$ and $\widetilde{Q}=P_0acP_1ac\dots acP_k$, $k>0$.

The so-fixed $\widetilde{Q}$ allows to derive that $P_i$ ($0\leqslant i\leqslant k$) are palindromes. Indeed, factor \eqref{eq:f_extension_Q} contains $cP_i$, $P_ic$, $aP_i$, $P_ia$ and therefore one proceeds along the same lines as in the Step $2$ proving that $Q = I(Q)$. Then $cI(Q)a = cQa = \sigma(\phi(C(p,\ell-p)))$ with $p$ even and $\sigma\in\mathfrak{S}_3$ such that  $\sigma(\leta) = c$, $\sigma(\letb) = b$ and $\sigma(\letc) = a$. Then Christoffel word $C(p,\ell-p)$ has the form $C(p,\ell-p)=\leta \hat{P}_k\letb\leta\dots\letb\leta\hat{P}_1\letb\leta\hat{P}_0\letb $ where $\hat{P}_i$ ($0\leqslant i\leqslant k$) are palindromes. Moreover, it appears that $\leta \hat{P}_i \letb$ ($0\leqslant i\leqslant k$) are Christoffel words. This follows from the Christoffel tree construction: by reading $C(p,\ell-p)$ from left to right each factor $\letb\leta$ designates the last and the first letters of Christoffel factors.

 The condition for the words $cP_ia$ ($0\leqslant i\leqslant k$), to contain palindromes follows from the construction of the map $\phi$ that the number of occurrences of the letters $c$ and $a$ in $cP_ia$ is even. It implies the even number of occurrences of the letter $\leta$ in all Christoffel words $\leta \hat{P}_i \letb$, $0\leqslant i\leqslant k$. But this is in contradiction with the inductive construction of Christoffel words. As a result of the above contradictions we are left with $\widetilde{Q} = Q$.

As soon as $Q$ is preserved by extending $a f^{(\ell)}_{m}(Q)c$ to the right, {\it i.e.} that $\widetilde{Q} = Q$ in \eqref{eq:f_extension_Q}, the only possibility would be to change the letters after $Q$. But because $\com(acQ) = \beta$ the only possibilities for
\begin{equation}\label{eq:f_extension_Q_plus}
    a f_m^{(\ell)}(Q) \underbrace{cQa\dots cQa}_{n\geqslant 0 \text{ factors } cQa} c Q x \subset [w]
\end{equation}
would be $x = a$ or $x = c$. For the former case one also fixes the next letter $cQx \subset cQac$ by analysing the $3$-spectrum along the same lines as above, which returns to the \eqref{eq:f_extension_Q} and repeating the analysis. Another case $x = c$ corresponds to a sequence $\chg\chb^{(n+1)l-1}\cha\subset [W_{\ell}]$, which leads to a factor $\bar{f}^{(\ell)}_n(Q)$ \eqref{eq:bfactor_Q} ``merged'' to $f^{(\ell)}_m(Q)$ as follows
\begin{equation*}
    F^{(\ell)}_{m,n}(Q)cQc = cQc\underbrace{aQc\dots aQc}_{m\geqslant 0 \text{ factors } aQc} aQa \underbrace{cQa\dots cQa}_{n\geqslant 0 \text{ factors } cQa}cQc \subset [w].
\end{equation*}
By continuing the above arguments one arrives at factors of the form
\begin{equation}\label{eq:representative}
    F^{(\ell)}_{m_1,\dots,m_{2t}}(Q) := F^{(\ell)}_{m_1,m_{2}}(Q) F^{(\ell)}_{m_3,m_{4}}(Q)\dots F^{(\ell)}_{m_{2t-1},m_{2t}}(Q) \subset[w],\quad m_1,\dots, m_{2t}\geqslant 0,
\end{equation}
and thus $[w] = [F^{(\ell)}_{m_1,\dots,m_{2t}}(Q)]$ for some $m_1,\dots, m_{2t}\geqslant 0$ where $t\geqslant 1$. But this means exactly $[w] = [\sigma\circ\phi(T)]$ for some $T\in \mathcal{C}^{\text{tw}}$. 

\paragraph{Step 4.} To finish the proof let us verify that for all primitive $\widetilde{w}\in\mathcal{C}^{\text{tw}}$ one has $[\phi(\widetilde{w})]\in[\cM_{3}]$, in particular that all $\# \mathsf{spec}_{n}[\phi(\widetilde{w})] = 3$ for $n<|\widetilde{w}|$. Recall that $\widetilde{w}$ is constructed from a Christoffel word $w\in\mathcal{C}$ for which we have $[\phi(w^{p})]\in[\cM_{3}]$ for any $p\in\mathbb{N}$ (by the Lemmas \ref{thm:BsubM} and \ref{lem:building_blocks}). 

Let us prove that $\mathsf{spec}_{n}[\phi(\widetilde{w})]\subset \mathsf{spec}_{n}[\phi(w^p)]$ for all $n$ satisfying $k|w| < n < (k+1)|w|\leqslant p|w|$ at any $k\geqslant 1$. If a factor $W^{\text{tw}}\subset \phi(\widetilde{w})$ obtained from a factor $W\subset w$ of the length $n$ has the form $W^{\text{tw}}=Q^\prime x_1x_2 Q\dots Q x_{i}x_{i+1} Q^{\prime\prime} $, $x_1,\dots x_{i+1}=\leta,\letc$ (with $0\leqslant |Q^\prime|,|Q^{\prime\prime}|<|Q|$) then $\com(W^{\text{tw}})=\com (W)$. Another possible forms of factors are $W_1^{\text{tw}}=x_0 Q x_1x_2Q\dots Q x_{i}x_{i+1} Q^{\prime}$ or $W_2^{\text{tw}}= Q^{\prime} x_1x_2Q\dots Q x_{i}x_{i+1} Q x_0 $. We analyze the case $W^{\text{tw}}_1$ ($W^{\text{tw}}_2$ is considered along the same lines). If $x_0=\leta$ then it coincides with the first letter of the respective factor $W$ and  $\com(W^{\text{tw}})=\com (W)$. If $x_0=\letc$ then by construction there is a factor in $w$ of the form $W^\prime=Q^{\prime} \letc\leta Q\dots Q \letc\leta Q \letc$ such that $\com(W^{\text{tw}})=\com (W^\prime)$. With this at hand it is straightforward to conclude that also $\mathsf{spec}_{n}[\phi(\widetilde{w})]\subset \mathsf{spec}_{n}[\phi(w^p)]$ and $\# \mathsf{spec}_{n}[\phi(\widetilde{w})] = 3$ for  $k|w| < n < (k+1)|w| \leqslant p|w|$ at any $k\geqslant 1$. Along the same lines one can show the reverse inclusion $\mathsf{spec}_{n}[\phi(w^p)]\subset\mathsf{spec}_{n}[\phi(\widetilde{w})]$, thus by Theorem \ref{thm:main_1} $\#\mathsf{spec}_{n}[\phi(\widetilde{w})] = 3$ for all $n$ satisfying $k|w| < n < (k+1)|w|\leqslant p|w|$ at any $k\geqslant 1$.

For $n=k|w|<p\ell$ it is straightforward to find 
\begin{multline*}
    \mathsf{spec}_{n}[\phi(\widetilde{w})] = \{ 2\cdot\leta+(k-1)\,\leta+(k-1)\,\letc+k\,\com(Q), \\ k\,\leta+k\,\letc+k\,\com(Q),2\cdot\letc+(k-1)\,\leta+(k-1)\,\letc+k\,\com(Q)\}\,.
\end{multline*}

\end{proof}

\providecommand{\href}[2]{#2}\begingroup\raggedright\endgroup


\end{document}